\documentclass[12pt]{amsart}
\usepackage{amsfonts}
\usepackage{amsfonts,latexsym,rawfonts,amsmath,amssymb,amsthm}
\usepackage[plainpages=false]{hyperref}

\usepackage{graphicx}

\RequirePackage{color}

\numberwithin{equation}{section}

\newcommand{\beq}{\begin{equation}}
\newcommand{\eeq}{\end{equation}}
\newcommand{\beqs}{\begin{eqnarray*}}
\newcommand{\eeqs}{\end{eqnarray*}}
\newcommand{\beqn}{\begin{eqnarray}}
\newcommand{\eeqn}{\end{eqnarray}}
\newcommand{\beqa}{\begin{array}}
\newcommand{\eeqa}{\end{array}}

\def\C{\mathcal C}

\newtheorem{Proposition}{Proposition}[section]
\newtheorem{Theorem}[Proposition]{Theorem}
\newtheorem{Lemma}[Proposition]{Lemma}

\newtheorem{Corollary}[Proposition]{Corollary}

\title{Continuity estimates for the gradient of solutions to the Monge-Amp\`ere equation with natural boundary conditions}

\begin{document}

\address{Huaiyu Jian: Department of Mathematical Sciences, Tsinghua University, Beijing 100084, China.}

\address{Ruixuan Zhu: Department of Mathematical Sciences, Tsinghua University, Beijing 100084, China.}

\email{
hjian@tsinghua.edu.cn; zhurx19@mails.tsinghua.edu.cn}

\thanks{This work was supported by NSFC 12141103.}


\bibliographystyle{plain}

\maketitle

\baselineskip=15.8pt
\parskip=3pt

\centerline{\bf   Huaiyu Jian \ \ \ \  Ruixuan Zhu}

\centerline{Department of Mathematical Sciences, Tsinghua University, Beijing 100084, China}

\vskip20pt

\noindent{\bf Abstract}:
  We study the first derivative estimates for solutions to Monge-Amp\`ere equations in terms of modulus of continuity. As a result, we establish the optimal global log-Lipschitz continuity for the gradient of solutions to the Monge-Amp\`ere equation with natural boundary condition.


 \vskip20pt
 \noindent{\bf Key Words:} The a priori estimate, Monge-Amp\`ere type equation, natural boundary condition.
 \vskip20pt

\noindent {\bf AMS Mathematics Subject Classification}: 35J96, 35J60, 35J75, 35Q82.

\vskip20pt

\noindent {\bf  Running head:}  Continuity estimates for Monge-Amp\`ere equation

\vskip20pt

\baselineskip=15.8pt
\parskip=3pt

\newpage

\centerline{\bf   Continuity estimates for the gradient of solutions}
\centerline{\bf to the Monge-Amp\`ere equation with natural boundary condition}

 \vskip10pt

\centerline{\bf   Huaiyu Jian \ \ \ \  Ruixuan Zhu }

\maketitle

\baselineskip=15.8pt
\parskip=3.0pt

\section{Introduction}
In this paper, we study the global first derivative estimates for solutions to the Monge-Amp\`ere equation with natural boundary condition:
\begin{equation}\label{1}
  \left\{\begin{aligned}
  &\det D^2u =\frac{f}{g(Du)} & & \text{ in }\Omega,\\
  &Du(\Omega)=\Omega^*.
\end{aligned}
\right.\end{equation}
We will establish the {\sl a priori} estimate for $Du$ on $\overline{\Omega}$ in terms of the modulus of continuity of $\log (f/g)$. As a result, we will obtain the log-Lipschitz regularity of the gradient of solutions to problem (1.1) with the optimal growth condition for $\log(f/g)$.

The Monge-Amp\`ere equation with natural boundary condition arises in the problem of optimal transportation. It has extensive applications in vast areas, including image processing, economics, meteorology, and design problems.
See\cite{EV, PF, RR, V03, V09} for example. The existence and uniqueness of convex solutions to this problem were obtained in 1991 by Brenier~\cite{br}. The regularity theories for equation~\eqref{1} were investigated by many authors~\cite{ca1,ca2,ca4,ca3,w1,jw,clw}. When $g=1$ and $f$ is strictly positive and bounded, Caffarelli proved the interior $C^{1,\alpha}$ regularity for the solution to the first equation of (1.1) in~\cite{ca2} and then the global $C^{1,\alpha}$ regularity for problem (1.1) in~\cite{ca4}. If additionally assume that $f$ is continuous, Caffarelli proved the interior $W^{2,p}$ regularity for the first equation of (1.1) in~\cite{ca1}. Moreover, Wang~\cite{w1} proved that if $f$ is Dini continuous, then the solution should be in $C^2$. In 1996, Caffarelli established~\cite{ca3} the global $C^{2,\alpha'}$ regularity of the equation~\eqref{1} with $\alpha'<\alpha$ when $f$ is $C^{\alpha}$ continuous and $ \Omega, \Omega^*$ are uniformly convex $C^2$ bound domains. Recently, Chen, Liu and Wang~\cite{clw} improved this result and proved the global $C^{2,\alpha}$ (Schauder) regularity for $f\in C^{\alpha}$ and $C^{1,1}$ convex domains $\Omega$ and $\Omega^*$.

A simple proof for the interior Schauder estimates for elliptic equations was first raised by Wang~\cite{ws}, who together with Jian~\cite{jw} used the method to derive the interior Schauder estimates
 as well as the interior log-Lipschitz continuity of gradient of solutions to the Monge-Amp\`ere equation $\det D^2u=f$, which is the special case of $g\equiv 1$. The argument was then also used by other authors to investigate the regularity theory for Monge-Amp\`ere equations; see~\cite{clw2, clw, cheng2, cheng1} and page 130 and 136 of~\cite{f17} for example. The global Schauder regularity of the Dirichlet problem was obtained in~\cite{tw08, s13} and the oblique problem was obtained in~\cite{jt22, jt23} for the Monge-Amp\`ere equation $\det D^2u=f$.
 In this paper we are mainly concerned with the global log-Lipschitz estimates of gradient of solutions to problem (1.1). We will extend the method in~\cite{jw} and use the related ideas in
~\cite{clw2, clw} to study the global regularity for problem~\eqref{1} without the Dini continuity of $f$ and $g$. As a result, we obtain the global log-Lipschitz continuity of gradient of solutions to problem (1.1)
 under the optimal growth condition for $\log(f/g)$.

 To state our main result, we need the concept of modulus of continuity and modulus of convexity.
 For a continuous function $F:\,D\to \mathbb{R}$, its modulus of continuity is defined as
$$w_F(r)=\sup\{|F(x)-F(y)|:x,y\in D\text{ and } |x-y|\le r\},$$
 and it is called Dini continuous if
$$\int_0^1 \frac{w_F(r)}{r}\,\mathrm{d}r<+\infty.$$
 The modulus of convexity for a $C^1$ convex function $u$ is defined as
$$C_u(t):=\inf\{u(x)-T_z(x):\; |x-z|>t\},$$
where $t>0$ and $T_z(x) $ is the tangent plane of $u$ at $z$. We will use the following assumption:
\begin{equation}\label{1.2}
\Omega \text{ and } \Omega^* \text{ are } C^{1,1} \text{ bounded and convex domains in } \mathbb{R}^n;
\end{equation}
\begin{equation}\label{1.3}
 f\in C(\bar \Omega), \ \ g\in C(\bar \Omega^*) ;
\end{equation} and there are two positive constants $\lambda$ and $\Lambda$ such that
 \begin{equation}\label{1.4}
 0<\lambda\le f(x),g(y) \le\Lambda<+\infty, \ \ \forall x\in \Omega, \forall y\in \Omega^*.
\end{equation}
Under these assumptions, if $u$ is a convex solution to problem~\eqref{1}, then $u\in C^{1, \alpha}(\bar \Omega)$ by Caffarelli' s work, as mentioned before. We will prove the following theorem which indicates the continuity estimate of $Du$ in dependence of the modulus of convexity for $u$ and the modulus of continuity of $f$ and $g$.

\begin{Theorem}\label{1.1}
 With assumption (1.2)-(1.4), let $u$ be a convex solution to problem~\eqref{1}. Then we have
 \begin{equation}\label{1.5}
   |Du(x)-Du(y)|\le Cd[1+e^{2\theta\psi(d)}]\quad\text{ on }\overline{\Omega}, \ \ \forall x, y\in \bar \Omega .
 \end{equation}
 Here $\theta$ is positive constant, $d=|x-y|$, $\psi(d)=\int_d^1\frac{w_{\log f/g}(r)}{r}dr$.
 The constant $C$ depends on $n$, $\lambda$, $\Lambda$, the modulus of convexity for $u$ and the modulus of continuity of $f$ and $g$.
\end{Theorem}

The next corollary follows immediately from Theorem 1.1, suggesting the global log-Lipschitz regularity for the gradient of solutions to problem (1.1) when the oscillation of the right-hand side function satisfies certain growth conditions.
 \begin{Corollary} \label {1.2}
Let $u$ be a convex solution to the equation~\eqref{1}. With the same assumptions as in Theorem 1.1, if there exist two positive constants $\delta$ and $c$ such that, for any $r<\delta$,
$$w_{\log f/g}(r)\le \frac{c}{|\log r|},$$
then $Du$ is globally log-Lipschitz continuous, i.e.,
$$|Du(x)-Du(y)|\le Cd(1+|\log d|)\text{ for any }x,y\in\overline{\Omega}.$$
Here $d=|x-y|$, the constant $C>0$ depend only on $n$, $\lambda$, $\Lambda$, $\Omega$, $\Omega^*$ and the modulus of continuity of $f$ and $g$.
\end{Corollary}

This estimate in Theorem 1.1 is optimal in the way that if $w_{\log f/g}(r)\ge C/|\log r|$ for some large constant $C$, then the log-Lipschitz continuity of $Du$ may not hold. One can refer to Section 4 in Jian and Wang's paper~\cite{jw} for a counterexample.

 In Section 2 we will present some basic results, including the sub-level sets of solutions to~\eqref{1} and the interior gradient estimates for the Monge-Amp\`ere equation with general right-hand side functions, which are useful in our proof for the main theorem. In the first two subsections of Section 3, we introduce the concept of boundary sections and derive a comparison lemma for the solution to (1.1) and the auxiliary solutions in shrinking boundary sections. Section 3.3 is devoted to the proof of Theorem 3.1, which results in Theorem 1.1 directly.

\section{Basic properties}
 Let $u$ be a differentiable convex function defined on $\Omega$. We extend $u$ to a convex function $\tilde{u}$ in $\mathbb{R}^n$ such that
$$\tilde{u}(x)=\sup\{L(x): L \text{ is any linear support function of } u \text{ in }\Omega\}.$$
We still denote $\tilde{u}$ by $u$.
Let $x_0$ be any point in $\Omega$. Define the sub-level set (or the section of) $u$ at (based point) $x_0$ with (height) $h>0$ as,
$$S_h[u](x_0)=\{y\in\Omega: u(y)<u(x_0)+(y-x_0)\cdot Du(x_0)+h\}.$$
It is the corresponding subset of $\Omega$ where the graph of $u$ is below the plane deduced by moving up the tangential plane of $u$ at point $x_0$ by $h$ height.
 If no confusion happens, we sometimes write $S_h[u](x_0)$ as $S_h$. Then $S_h$ is a convex domain for any $h>0$.
The centered section of $u$ at $x_0$ with height $h$ is the convex set
$$S_h^c[u](x_0)=\{y\in\mathbb{R}^n: u(y)< \ell(y)\},
$$
where $\ell$ is an affine function such that $\ell(x_0)=u(x_0)+h$, and $x_0$ is the center of mass of $S_h^c[u](x_0)$. By~\cite{ca3}, such $\ell$ always exist for any $x_0$ and $h>0$ small.

For two convex body $E$ and $F$, we denote $E\sim F$ if there are two positive constants $C_0$ and $C_1$ such that
$$C_0F\subset E\subset C_1F,$$
where the scaling is defined with respect to some balanced point of $F$. Here we say $x\in F$ is a balanced point of $F$ if there is a constant $C>0$ such that for any $y\in \partial F$, there exists a point $-t(y-x)\in \partial F$ with some $1/C\le t\le C$. It is well-known that for any convex body $E$, one can find a unimodular affine transform $T$ to make $TE\sim B$ for some ball $B$ centered at some balanced point of $E$. Similarly, we say that an affine transform $T$ normalizes a sub-level set $S_h$ of $u$ if $TS_h\sim B$ for some ball $B$. Moreover, if
$$B_{C_0}(y_0)\subset TS_h\subset B_{C_1}(y_0)$$
with $C_1/C_0\le C$ for some universal constant $C$ and some point $y_0$, then $TS_h$ is called to {\sl have a good shape}.

 Under the assumptions of Theorem 1.1, we have the following properties of sections and centered sections of the solution $u$ to problem~\eqref{1}. See~\cite{ca3, clw} for the details.
\begin{enumerate}
 \item Let $x_1$, $x_2$ be two points in $\Omega$, $0<t_1<t_2<1$, and $x_2\in t_1S_h^c[u](x_1)$. Then there exists some constant $s$ depending on $t_1$, $t_2$ such that
 $$S^c_{sh}[u](x_2)\subset t_2 S^c_{h}[u](x_1).$$
 \item The volume of $S_h^c[u](x_0)\approx h^{n/2}$ for any small $h>0$ and any point $x_0\in \overline{\Omega}$.
 \item There exists some constant $b>0$ such that
 $$ S_{bh}^c[u](x_0)\cap \Omega\subset S_{h}[u](x_0)\subset S_{b^{-1}h}^c[u](x_0)\cap \Omega$$
 for any $x_0\in\partial \Omega$ and $h>0$ small enough.
 \item For any small $\epsilon>0$, there is some constant $C>0$ depending on $\epsilon$ and $\Omega,\Omega^*,f$ and $g$, such that
 $$B_{C^{-1}h^{\frac{1}{2}+\epsilon}}(x_0)\subset S^c_h[u](x_0)\subset B_{Ch^{\frac{1}{2}-\epsilon}}(x_0)$$ for any $x_0\in \overline{\Omega}$ and $h>0$ small enough.
\end{enumerate}

Properties (1), (2) and (3) guarantee the geometric property of $S_h$. Property (4) is due to the $W^{2,p}$ regularity of $u$ in equation~\eqref{1} when $f$ and $g$ are continuous.

 In the end of this section, we state
the interior estimate for the gradient of solutions to
\begin{equation}\label{eq-1}
   \det D^2 u=B(x,u(x),Du(x))\quad\text{ in } B_1(0)
\end{equation}
where we assume $B \in C(\mathbb{R}^n\times\mathbb{R}\times \mathbb{R}^n) $ and there exist positive constants $\lambda$ and $\Lambda$ such that for any $C^1$ functions $u$, $B_u(x):=B(x, u(x), Du(x))$ satisfy
\begin{equation}\label{cdn-1}
 0<\lambda\le B_u\le \Lambda<+\infty\quad\text{ on } B_1(0).
\end{equation}
 If $u$ is a strictly convex solution of~\eqref{eq-1}, then $u$ is in $C^{1,\alpha}(B_1(0))$ due to the regularity theory of Monge-Amp\'ere equations (see~\cite{ca2} for a reference).
  Hence the right-hand side of equation~\eqref{eq-1} can be viewed as a continuous positive function on $B_1(0)$. Then by Theorem 1 in~\cite{ca1}, $u\in W_{loc}^{2,p}$ for all $1<p<+\infty$ and thus in $C^{1,\alpha}$ for all $\alpha<1$. Once the modulus of continuity of $B_u$ is controlled, $u$ has some higher kind of regularity than $W^{2,p}$. To be more explicit, we have the following theorem.

\begin{Theorem}\label{2.1}
Let $u$ be a strictly convex solution to the Monge-Amp\`ere equation~\eqref{eq-1}, where $B_u\in \C(B_1(0))$ and satisfies condition~\eqref{cdn-1}. Then we have
\begin{equation}
 |Du(x)-Du(y)|\le Cd[1+e^{2\theta\psi(d)}], \ \ \forall x,y\in B_{1/2}(0) .
\end{equation}
 Here $\theta$ is positive constant, $d=|x-y|$, $\psi(d)=\int_d^1\frac{w_{\log B_u}(r)}{r}dr$. The constant $C$ depends on $n$, $\lambda$, $\Lambda$, the modulus of convexity for $u$ and $w_{B_u}$.
\end{Theorem}
\begin{proof}
 It follows directly from Theorem 2 in~\cite{jw} if we take $f(x)$ there as $B_u(x).$
\end{proof}

\section{Global estimates}
 This section is devoted to the proofs of Theorem 1.1 and Corollary 1.2. First,
 Corollary 1.2 follows from Theorem 1.1 by a direct computation. To prove Theorem 1.1,
 it is sufficient to prove the case of $g\equiv 1$, that is Theorem 3.1 below. Then the result in Theorem 1.1 follows directly from
 Theorem 3.1 and the global $C^{1, \alpha}$ regularity for (1.1).

\begin{Theorem}\label{thm3}
Let $u$ be a convex solution to the equation
\begin{equation}\label{1-1}
 \left\{\begin{aligned}
   &\det D^2u =f & & \text{ in }\Omega,\\
   &Du(\Omega)=\Omega^*,
\end{aligned}
\right.
\end{equation}
$f$, $\Omega$ and $\Omega^*$ as in Theorem 1.1. Then we have
\begin{equation}\label{es-1}
 |Du(x)-Du(y)|\le Cd[1+e^{2\theta\psi(d)}], \ \ \forall x, y \in \overline{\Omega} .
\end{equation}
Here $\theta$ is positive constant, $d=|x-y|$, $\psi(d)=\int_d^1\frac{w_{\log f}(r)}{r}dr$. The constant $C$ depends on $n$, $\lambda$, $\Lambda$, the modulus of convexity of $u$ and $w_f$.
\end{Theorem}
From now on, we always assume $f, \Omega, \Omega^*$ and $u$ are the same as in Theorem 3.1.
We will prove the theorem in three steps. First we shall prove~\eqref{es-1} for $x,y\in \partial\Omega$, then for $x\in \partial\Omega$ and $y\in\overline{\Omega}$, and finally for
$x, y\in\overline{\Omega}$. For this purpose, we need to define the boundary sections $D_h$ for $u$, on which we will make iterations.

\subsection{Boundary Sections}
It's known from~\cite{clw} that $u$ is strictly convex and $u\in W^{2,p}(\overline{\Omega})$ for any $p>1$, thus $u$ is in $C^{1,1-\epsilon}(\overline{\Omega})$ for any $\epsilon>0$. The uniform obliqueness of
 problem (3.1) is guaranteed by Lemma 1.1 in~\cite{clw}, stated as below.
\begin{Lemma}\label{lem-ob}
 For any point $x_0\in\partial\Omega$ and $y_0=Du(x_0)$ at $\partial\Omega^*$, suppose that $\nu$ is the unit inner normal vector of $\partial\Omega$ at $x_0$, and $\nu^*$ is the unit inner normal vector of $\partial\Omega^*$ at $y_0$. Then under the same assumption of Theorem 3.1, there exists a universal constant $\delta>0$ such that
$$\nu\cdot \nu^*>\delta>0$$
\end{Lemma}

Suppose that $0\in\partial\Omega$, $u(0)=0$, $Du(0)=0$, and $u$ is extended to a convex function in $\mathbb{R}^n$. Let $S_h^0[u](0)=\{x\in\mathbb{R}^n:\,u(x)<h\}$.
Suppose also that $\partial\Omega$ and $\partial\Omega^*$ share the same inner normal vector $(0,\cdots,0,1)$ at 0, and they both lie on the upper half space of $\mathbb{R}^n$.
Let $h>0$ be sufficiently small, define
\begin{equation}\label{3.3}D_h^+[u](0)=S_h^0[u](0)\cap \{x_n\ge a_h\},
\end{equation}
where \begin{equation}\label{3.4}
a_h=\inf_a\{S_h^0\cap \{x_n>a\}\subset\Omega\}.
\end{equation}
So we have $a_h\to 0$ when $h\to 0$. Let $D_h^-[u](0)$ be the reflection of $D_h^+[u](0)$ with respect to the plane $\{x_n=a_h\}$. Define $D_h[u](0)$, the boundary section of $u$ at $0$, to be the union of $D_h^+[u](0)$ and $D_h^-[u](0)$, which will be shorten as $D_h$.
Since $u_n:=\frac{\partial u}{\partial x_n}>0$ in $\Omega$, the definition is well-defined, see~\cite{clw2,clw} for a reference. We know that $D_h$ is a convex domain, and $D_h$ shrinks to the point $0$ when $h\to 0$. There are some basic properties for $D_h$, making it behave like centered sections.

\begin{Lemma}\label{prop}
 $D_h$ is a convex domain in $\mathbb{R}^n$, and has the following properties.
\begin{enumerate}
 \item  The volume of $D_h$ satisfies the inequality
 $$1/C h^{n/2}\le |D_h|\le Ch^{n/2}$$
 for some constant $C>0$ (independent of $u$ and $h$), and for any $\epsilon>0$, there is a constant $C_{\epsilon}>0$, depending on $\epsilon$, $f$, $\Omega$ and $\Omega^*$, such that for sufficiently small $h>0$,
 $$a_h\le C_{\epsilon} h^{1-\epsilon}.$$
 \item We have for sufficiently small $h>0$,
 \begin{equation*}
   S_{1/Ch}^c(0)\subset D_h[u](0)\subset S_{Ch}^c(0).
 \end{equation*}
 \item $0$ is a balanced point of $D_h[u](0)$, and for any $\epsilon>0$ small, there is a constant $C_{\epsilon}>0$ such that for sufficiently small $h>0$,
 \begin{equation*}
   B_{1/C_{\epsilon}h^{1/2+\epsilon}}(0)\subset D_h[u](0)\subset B_{C_{\epsilon}h^{1/2-\epsilon}}(0).
 \end{equation*}
 \end{enumerate}
\end{Lemma}
\begin{proof}
The estimates for $a_h$, $|D_h|$ follow from Lemma 6.1 in~\cite{clw}. Since $a_h\le Ch^{1-\epsilon}$ and $$B_{1/Ch^{1/2+\epsilon}}(a_he_n)\subset D_h\subset B_{Ch^{1/2-\epsilon}}(a_he_n)$$
where $e_n=(0,\cdots,0,1)$,
we have (3) for sufficiently small $h$. By definition, $D_h^+\subset S_h^0\cap \mathbb{R}^n_+$, and by (6.8) in~\cite{clw} we have
\begin{equation}\label{fm-3}
 S_{1/Ch}^c(0)\cap \mathbb{R}^n_+\subset S_h^0\cap\mathbb{R}^n_+\subset S_{Ch}^c(0)\cap \mathbb{R}^n_+.
\end{equation}
Since $0$ is the balanced point of $S_{Ch}^c(0)$, and $a_h=o(h^{\frac{1}{2}+\epsilon})$ for any $\epsilon\in (0, \frac{1}{4})$ when $h\to 0$, we have $S_{1/Ch}^c(0)\subset D_h(0)\subset S_{Ch}^c(0)$. Thus $0$ is a balanced point of $D_h$.
\end{proof}

For any point $x_0\in\partial\Omega$, $Du(x_0)\in\partial\Omega^*$, generally $\partial\Omega$ and $\partial\Omega^*$ don't share a common normal vector at $x_0$. Assume that the inner normal vectors of $\partial \Omega$ at $x_0$ and that of $\partial \Omega^*$ at $Du(x_0)$ are $n$ and $n^*$. Due to Lemma~\ref{lem1} below, one can find a linear transform $A$ that makes the two inner normal vectors coincide.

\begin{Lemma}\label{lem1}
Given $\delta>0$, there exists a constant $C>0$ such that, for any unit vector $n$, $n^*\in\mathbb{R}^n$, with $n\cdot n^*>\delta>0$, one can find a unimodular matrix $A$ with $\det A=1$, such that $An$ and $An^*$ both lie on the positive $x_n$ axis, and $1/C\le ||A||,||A^{-1}||\le C$.
\end{Lemma}
\begin{proof}
One can find an orthogonal matrix $Q$ with $\det Q=1$, such that $Qn=e_n=(0,\cdots,0,1)^t$, and $Q$ maps $n^*$ to a vector on the plane spanned by $x_1$ and $x_n$ axis. Assume that $Qn^*=(a,0,\cdots,0,b)^t$. Since
$n\cdot n^*>\delta >0$ we have $b>0$. Let
$$P=\begin{pmatrix}
 1& &\\
 &\ddots &\\
 \frac{a}{b}& &1
 \end{pmatrix}
 .\quad \text{Then } P^{-t}:=(P^{-1})^t=\begin{pmatrix}
 1& &-\frac{a}{b}\\
 &\ddots &\\
 & &1
\end{pmatrix},
$$
 $Pe_n=e_n$, and $P^{-t}(a,0,\cdots,0,b)^t=be_n$. Since $|a/b|=\tan(\arccos(n\cdot n^*))$,
we have $|a/b|\le \sqrt{(1/\delta^2-1)}$. Note that $\det P=1$, $1/C\le ||P^{-1}||, ||P||\le C$. Take $A=PQ$ to finish the proof.
\end{proof}

With $A$ given by Lemma 3.4, we define
$$
\begin{aligned}
&\hat{u}(\hat{x})=u(A^{-1}\hat{x}+x_0)-u(x_0)-<Du(x_0), A^{-1}\hat{x}>,\\
&\hat{\Omega}=A(\Omega-x_0),\\
&\hat{\Omega}^*=A^{-t}(\Omega^*-Du(x_0)).\\
\end{aligned}
$$
Then $\hat{u}$ solves the problem
$$\left\{
\begin{aligned}
&\det D^2\hat{u}(\hat{x})=f(A^{-1}\hat{x}+x_0) \quad \text{ in }\hat{\Omega},\\
&D\hat{u}(\hat{\Omega})=\hat{\Omega}^*.
\end{aligned}\right.
$$
And $\hat{u}(0)=0$, $D\hat{u}(0)=0$.
Define $D_h^+[\hat{u}](0)$ and $\hat{a}_h$ as (3.3)-(3.4). Let $D_h^-[\hat{u}](0)$ be the reflection of $D_h^+[\hat{u}](0)$ w.r.t. $\{\hat{x}_n=\hat{a}_h\}$,
and $D_h[\hat{u}](0)=D_h^+[\hat{u}](0)\cup D_h^-[\hat{u}](0)$.

We define
$$D_h[u](x_0)=A^{-1}D_h[\hat{u}](0)+x_0.$$
Note that this definition depends on the choice of $A$. While the choice of $A$ is not unique, we will see in the proof that this does not affect the estimations since the norm of $A$ and $A^{-1}$ are globally bounded. Obviously, properties of $D_h[u](x_0)$ inherit directly from all the propositions in Lemma~\ref{prop}, so we will use them without stating them again.

\subsection{A Comparison Lemma}
We shall use the following comparison lemma to compute the difference of $u$ and the standard solution in the boundary sections shrinking to zero.

\begin{Lemma}\label{lem-c}
Assume $\Omega \subset \mathbb{R}^n_{+}$ with $0\in \partial \Omega$. Let $v\in C^{1,1-\epsilon}(\overline{\Omega})$ be a convex solution to $\det D^2v=f$ in $\Omega$ with $f(0)=1$,
$v(0)=0$ and $ Dv(0)=0$. Linearly extend $v$ to $\mathbb{R}^n_{+}$ and evenly reflect it through $\{x_n=0\}$ to extend the definition of $v$ to $\mathbb{R}^n$.
Suppose that $\partial\Omega$ and $\partial\Omega^*$ can be expressed as $x_n=\rho(x')$ and $x_n=\rho^*(x')$ near $0$, with $\rho,\rho^*\in C^{1,1}$, $\rho(0)=\rho^*(0)=0$, and $D_{x'}\rho(0)=D_{x'}\rho^*(0)=0$.
Assume that $B_1\subset D_L:= D_L[v](0)\subset B_n$ for some $L>0$. Let $\gamma=||Dv||_{C^{0,1-\epsilon}(D_L)}$, $\delta_1=\sup_{D_L}|f-1|$, and $\delta_2=[D\rho]_{C^{0,1}(P_{x'}(D_L))}$, $\delta_2^*=[D\rho^*]_{C^{0,1}(P_{x'}(Dv(D_L)))}$.
Here and below, $P_{x'}$ denotes the projection map from $\mathbb{R}^n$ to $\{x_n=0\}$.

If $w\in C^2$ is the convex solution of $\det D^2w=1$ in $D_L$ with $w=L$ on $\partial D_L$,
then we have $$|v-w|\le C(\delta_1+\delta_2^*\gamma^2+\delta_2^{1-\epsilon}\gamma).$$
\end{Lemma}

\begin{proof}
Let $C_2=\partial D_L^+\cap \{x_n=a_L\}$, $C_1=\partial D_L^+\backslash C_2$, By symmetry we have $w_n=0$ on $C_2$. Let $$\hat{w}=(1-\delta_1)^{1/n}(w-L)+L.$$
Then
\begin{align*}
 \det D^2\hat{w}&=1-\delta_1\le\det D^2 v \quad\text{ in }D_L^+,\\
 \hat{w}&=v \quad\text{ on } C_1,\\
 \hat{w}_n&=0\le v_n\quad\text{ on }C_2.
\end{align*}
Thus by the comparison principle (see, for example, Lemma 2.4 in~\cite{jt23}) we have $v\le \hat{w}$ in $D_L^+$.

For $x\in C_2\cap \partial\Omega$, $x_n=\rho(x')\le n|D\rho(x')|\le n^2\delta_2$. Since $v\in C^{1,1-\epsilon}(\overline{\Omega})$, on $C_2$ we have
\begin{align*}
 v_n(z',z_n)&\le v_n(z',\rho(z'))+\gamma(n^2\delta_2)^{1-\epsilon}\\
 &\le \delta_2^*|D_{x'}v(z',\rho(z'))|^2+\gamma(n^2\delta_2)^{1-\epsilon}\\
 &\le C(n)(\delta_2^*\gamma^2+\delta_2^{1-\epsilon}\gamma)
\end{align*}
by $C^{1,1}$ regularity of $\partial\Omega^*$. Let
$$\check{w}=(1+\delta_1)^{1/n}(w-L)+L+C(n)(x_n-n)(\delta_2^*\gamma^2+\delta_2^{1-\epsilon}\gamma).$$
Then
\begin{align*}
 \det D^2\check{w}&=1+\delta_1\ge \det D^2v\quad\text{ in }D_L^+,\\
 \check{w}&\le L=v\quad\text{ on }C_1,\\
 \check{w}_n&=C(n)(\delta_2^*\gamma^2+\delta_2^{1-\epsilon}\gamma)\ge v_n\quad\text{ on }C_2.
\end{align*}
Again by comparison principle, $v\ge \check{w}$ in $D_L^+$. From the above arguments $|v-w|\le C(\delta_1+\delta_2^*\gamma^2+\delta_2^{1-\epsilon}\gamma)$ in $D_L^+$. In $D_L^-\cap \mathbb{R}^n_{+}$, let $x$ be the symmetric point of $z$ with respect to $C_2$, then
\begin{align*}
 |v(z)-w(z)|&\le |v(z)-v(x)|+C(\delta_1+\delta_2^*\gamma^2+\delta_2^{1-\epsilon}\gamma)+|w(x)-v(x)|\\
 &\le C(\gamma\delta_2+\delta_1+\delta_2^*\gamma^2+\delta_2^{1-\epsilon}\gamma)\\
 &\le C(\delta_1+\delta_2^*\gamma^2+\delta_2^{1-\epsilon}\gamma),
\end{align*}
by $C^{1,1}$ regularity of $\partial\Omega$ and with small $\delta_2$. In $D_L\cap\mathbb{R}^n_{-}$, we have also, by symmetry of $w$ and $v$,
$$|v(z)-w(z)|\le C(\delta_1+\delta_2^*\gamma^2+\delta_2^{1-\epsilon}\gamma).$$
Hence we have proved the inequality in $D_L$.
\end{proof}

\subsection{Proof of Theorem~\ref{thm3}}
At first, we quote three lemmas from~\cite{jw} that will be used throughout our proof. See Lemmas 2.2-2.4~\cite{jw} in for the details.

\begin{Lemma}\label{lem-1}
 Suppose that $w$ is a convex solution of $\det D^2w=1$ in a bounded convex domain $\Omega$, and $w$ is constant on $\partial\Omega$. $0$ is a balanced point of $\Omega$, $D^2w(0)$ is uniformly bounded. Then the domain $\Omega$ has a good shape.
\end{Lemma}
\begin{Lemma}\label{lem-2}
 Let $w_i$, $i=1,2$ be two convex solution of $\det D^2w_i=1$ in a bounded convex domain $\Omega$ in $\mathbb{R}^n$, with $B_1\subset\Omega\subset B_n$. And $||w_i||_{C^4(\overline{\Omega})}\le C_0$, $|w_1-w_2|\le \delta$ in $\Omega$ for some $\delta>0$. Then
 $$|D^k(w_1-w_2)|\le C\delta\quad\text{ in }\Omega'$$
 for and $\Omega' \subset\subset \Omega$ and $k=1,2,3$, where the constant $C$ depends on $C_0$ and $d(\Omega',\partial\Omega)$.
\end{Lemma}
\begin{Lemma}\label{lem-3}
   With the same $\Omega$ as in Lemma~\ref{lem-2}, let $w$ be a convex solution of $\det D^2w=1$ in $\Omega$ which is constant on $\partial\Omega$. Then for $\Omega' \subset\subset \Omega$,
   $||w||_{C^4(\Omega')}\le C$ for some constant $C$ depending on $n$ and $d(\Omega',\partial\Omega)$.
\end{Lemma}

Under the conditions of $f$, $\Omega$ and $\Omega^*$, problem~\eqref{1-1} has uniform strict obliqueness property as Lemma~\ref{lem-ob}, and its solution $u\in W^{2,p}(\overline{\Omega})$ for any $p>1$. Thus for any fixed $\epsilon>0$, $u\in C^{1,1-\epsilon}(\overline{\Omega})$. By Lemma~\ref{prop}, there exist a universal constant $h_0$ and $C$, with $h_0$ sufficiently small, such that for any $x\in \partial\Omega$,
$$ B_{Ch^{1/2+\epsilon}}(x)\subset D_h[u](x)\subset B_{Ch^{1/2-\epsilon}}(x), \ \  \forall h\in(0, h_0). $$
By the properties of the sections in Section 2 and Lemma~\ref{prop}, we have
$$D_{c_1h}[u](x)\subset D_{c_2h}[u](x)\subset D_h[u](x)$$
for some sufficiently small $c_2>c_1>0$. Without loss of generality, we assume that $c_1=\frac{1}{4}$ and $c_2=\frac{3}{4}$.
By property (1) and (3) in Section 2, we can also find a constant $\delta<1$ such that for any $z\in D_{1/4h}[u](x)$,
\begin{equation}\label{es-xz}
 D_{h\delta}[u](z)\subset D_{h}[u](x)\subset CD_{h\delta}[u](z).
\end{equation}

\subsubsection{Step 1} We will prove the inequality~\eqref{es-1} for $x,y\in\partial\Omega$ in this step. We shall follow the method of proof in~\cite{jw}, but establish all the estimates on iterative boundary sections $D_h$
instead of interior sections there.

For any $x_0\in \partial\Omega$ and $D_{h_0}[u](x_0)$ defined as before, by Lemma 3.4 and John's lemma (see, for example,~\cite{f17, g01}),
we may take an affine transform $A$ which normalizes $D_{h_0}[u](x_0)$ and also makes the inner normal vectors of $\Omega$ at $x_0$ and $\Omega^*$ at $Du(x_0)$ coincide. The norms of $A$ and $A^{-1}$ are bounded from up and below by Lemma~\ref{lem1}. Let
$$\tilde{u}(\tilde{x})=\frac{(\det A)^{2/n}}{f(x_0)^{1/n}}(u(A^{-1}\tilde{x}+x_0)-u(x_0)-<Du(x_0),A^{-1}\tilde{x}>).$$
And $\tilde{u}$ sovles the problem:
$$\left\{\begin{aligned}
 \det D^2\tilde{u}&=\frac{f(A^{-1}\tilde{x}+x_0)}{f(x_0)}=:\tilde{f}\quad\text{ in } \tilde{\Omega},\\
 D\tilde{u}(\tilde{\Omega})&=\tilde{\Omega}^*.
\end{aligned}\right.$$
Then $0\in \partial\tilde{\Omega}$, $\tilde{u}(0)=0$, $D\tilde{u}(0)=0$. The inner normal vectors of $\partial\Omega$ and $\partial\Omega^*$ at $0$ are both $e_n$. We may choose $A$ such that $\frac{(\det A)^{2/n}}{f(x_0)^{1/n}}=h_0^{-1}$, then $D_1[\tilde{u}](0)=D_{h_0}[u](x_0)\sim B_1(0)$.

We shall prove that for $\tilde{z}\in D_{4^{-4}}[\tilde{u}](0)\cap\partial\tilde{\Omega}$, but $\tilde{z}\neq 0$,
\begin{equation}\label{fm0}
 |D\tilde{u}(\tilde{z})|\le C\tilde{d}(1+e^{2\theta\int_{\tilde{d}}^1 \frac{w_{\tilde{f}}(r)}{r}\,dr}),
\end{equation}
where $\tilde{d}=|\tilde{z}|$ and $\theta$ is a positive constant.
Once we have proved this inequality, for $z\in B_{r_0}(x_0)\cap\partial\Omega$ where $r_0$ is a universal constant and sufficiently small, we have
$$|A^{-t}\frac{(\det A)^{2/n}}{f(x_0)^{1/n}}(Du(z)-Du(x_0))|\le C|z-x_0|(1+e^{2\theta\int_{|A(z-x_0)|}^1\frac{w_{\tilde{f}}(r)}{r}dr}).$$
The constant on the left side is universally bounded. For the right-hand side, we have $w_{\tilde{f}}(r)\le Cw_{\log f}(Cr)$ as estimated
on page 619 in~\cite{jw}. Thus~\eqref{es-1} is proved for $x,y\in\partial\Omega$.

Now let's prove inequality~\eqref{fm0}. For the sake of brevity we omit all tildes. Consider a sequence of boundary sections of $u$ at $0$, $D_{4^{-k}}=D_{4^{-k}}[u](0)$ for $k\ge 0$. Let $w_k$ be the convex solution of
the problem:
$$\left\{
 \begin{aligned}
   &\det D^2w_k=1 \quad\text{ in }D_{4^{-k}},\\
   &w_k=4^{-k} \quad\text{ on }\partial D_{4^{-k}}.
 \end{aligned}
\right.$$
Let $P_k=\sqrt{D^2w_k(0)}$, i.e., $D^2w_k(0)=P_k^2,$ and $P_k$ is a symmetric matrix and its eigenvalues, denoted by $\lambda(P_k)$, equals to $\lambda(D^2w_k(0))^{1/2}$.

Since $D^2(w_k\circ P_k^{-1})(0)=I$, by Lemma~\ref{lem-1} we see that $P_kD_{4^{-k}}$ has a good shape, with $2^kP_kD_{4^{-k}}\sim B_1(0)$.
Take unimodular matrix $A_k$ as in Lemma~\ref{lem1} such that the inner normal vectors of $P_k\Omega$ and $P_k^{-t}\Omega^*$ at 0 coincide, then $1/C\le ||A_k||,||A_k^{-1}||\le C$.
Let $Q_k=A_k^{-t}P_k$, then $Q_kD_{4^{-k}}$ still has a good shape. Thus we have $2^kQ_kD_{4^{-k}} \sim B_1(0)$.

Let $\hat{w}_0^k=4^kw_k\circ(2^{-k}Q_k^{-1})$, $\hat{w}_1^k=4^kw_{k+1}\circ(2^{-k}Q_k^{-1})$, $\hat{u}^k=4^ku\circ(2^{-k}Q_k^{-1})$. Then
$$\begin{aligned}
&\left\{\begin{aligned}
   &\det D^2\hat{w}_0^k =1 & & \text{ in }D_1[\hat{u}^k](0),\\
   &\hat{w}_0^k =1 & & \text{ on }\partial D_1[\hat{u}^k](0).
\end{aligned}\right.\\
&\left\{\begin{aligned}
   &\det D^2\hat{w}_1^k =1 & & \text{ in }D_{1/4}[\hat{u}^k](0),\\
   &\hat{w}^k_1 =1/4 & & \text{ on } \partial D_{1/4}[\hat{u}^k](0).
\end{aligned}\right.\\
&\left\{\begin{aligned}
   &\det D^2\hat{u}^k =f\circ(2^{-k}Q_k^{-1}) & & \text{ in }D_1[\hat{u}^k](0),\\
   &\hat{u}^k =\hat{u}^k & & \text{ on }\partial D_1[\hat{u}^k](0).
\end{aligned}\right.
\end{aligned}$$
Since $D_1[\hat{u}^k](0)\sim B_1$, and $d(D_{3/4}[\hat{u}^k],\partial D_1[\hat{u}^k])\ge C$, we have $||\hat{w}_0^k||_{C^4(D_{3/4}[\hat{u}^k])}\le C$ by Lemma~\ref{lem-3}. Similarly we have $||\hat{w}_1^k||_{C^4(D_{3/16}[\hat{u}^k])}\le C$.

Let $\Omega_k=2^kQ_k\Omega$, $\Omega_k^*=2^{-k}Q_k^{-t}\Omega^*$. Their boundaries near 0 have local representatives $x_n=\rho_k(x')$ and $x_n=\rho_k^*(x')$, respectively. Let
$$\left\{\begin{aligned}
 &\delta_k=\sup_{z\in\partial\Omega}\sup_{D_{4^{-k}}[u](z)\cap\Omega}|f-f(z)|,\\
 &\epsilon_k=[D\rho_k]_{C^{0,1}(P_{x'}(D_1[\hat{u}^k](0)))},\\
 &\epsilon^*_k=[D\rho^*_0]_{C^{0,1}(P_{x'}(D\hat{u}^k(D_1[\hat{u}^k](0))))},\\
 &\gamma_k=||D\hat{u}^k||_{C^{0,1-\epsilon}(D_1[\hat{u}^k](0))}.
 \end{aligned}
 \right.
$$
Denote $\mu_k=\delta_k+\epsilon_k^*\gamma_k^2+\epsilon_k^{1-\epsilon}\gamma_k$. Applying Lemma~\ref{lem-c} to $\hat{u}^k$ on $D_1[\hat{u}^k]$ and $D_{1/4}[\hat{u}^k](0)$, we have
$$\begin{aligned}
 |\hat{w}_0^k-\hat{u}^k|&\le C\mu_k\quad\text{ in } D_1[\hat{u}^k](0),\\
 |\hat{w}^k_1-\hat{u}^k|&\le C\mu_k\quad\text{ in } D_{1/4}[\hat{u}^k](0).
\end{aligned}$$
Hence we have
$$|\hat{w}_0^k-\hat{w}^k_1|\le C\mu_k\quad\text{ in } D_{1/4}[\hat{u}^k](0).$$

Thus by Lemma~\ref{lem-2}
$$\begin{aligned}
|D\hat{w}^k_0-D\hat{w}^k_1|&\le C\mu_k\quad\text{ in } D_{1/16}[\hat{u}^k](0),\\
|D^2\hat{w}^k_0-D^2\hat{w}^k_1|&\le C\mu_k\quad\text{ in } D_{1/16}[\hat{u}^k](0),
\end{aligned}$$
and therefore in $P_kD_{4^{-k-2}}[u](0)$, by boundedness of $||A_k||$ and $||A_k^{-1}||$ we have
\begin{gather}\label{-}
|D(w_k\circ P_k^{-1})-D(w_{k+1}\circ P_k^{-1})|\le C2^{-k}\mu_k,\\
\label{--}
|D^2(w_k\circ P_k^{-1})-D^2(w_{k+1}\circ P_k^{-1})|\le C\mu_k.
\end{gather}

Now for $z\in D_{4^{-4}}[u](0)\cap\partial{\Omega}$ but $z\neq 0$, we have a $k$ such that
 $z\in \partial\Omega\cap (D_{4^{-k-3}}[u](0)\backslash D_{4^{-k-4}}[u](0))$. Notice that
$$\arraycolsep=1.5pt
\begin{array}{rcccccc}
   |Du(z)-Du(0)|\le |Dw_k(z)&-&Dw_k(0)|+|Dw_k(0)&-&Du(0)|+|Dw_k(z)&-&Du(z)|.\\
   & I_1 & & I_2 & & I_3 &
\end{array}$$
Assume that $\lambda_k$ is the maximum eigenvalue of $P_k$. Then using~\eqref{--} to iterate, we have
$$\begin{aligned}
I_1 &\le |Dw_{k-1}(z)-Dw_{k-1}(0)|+C\lambda_{k-1}^2\mu_{k-1}|z|\\
&\le |Dw_0(z)-Dw_0(0)|+C\sum_{i=0}^{k-1}\lambda_i^2\mu_i|z|\\
&\le C|z|(1+\sum_{i=0}^{k}\lambda_i^2\mu_i).
\end{aligned}
$$
Since $Du(0)=\lim_{k\to\infty} Du_k(0)$, by~\eqref{-} we have
\begin{equation}\label{cse}
I_2 \le \sum_{i=k}^{+\infty}|Dw_i(0)-Dw_{i+1}(0)|\le \sum_{i=k}^{+\infty} C2^{-i}\lambda_i\mu_i.
\end{equation}

Before we estimate $I_3$, we need to estimate $\lambda_i$ and $\mu_i$. Since
 $$|D^2(w_{k+1}\circ P_k^{-1})(0)-I|=|D^2(w_{k+1}\circ P_k^{-1})(0)-D^2(w_{k}\circ P_k^{-1})(0)|\le C\mu_k,$$
 and $D^2(w_{k+1}\circ P_{k+1}^{-1})(0)=I$, if we write $P_{k+1}=T_kP_k$
 for some matrix $T_k$, then we have $|T_k^tT_k-I|\le C\mu_k$. Thus
  $\lambda_{max}(T_k)\le 1+\theta \mu_k$. Hence by the iteration $ P_k=T_{k-1}T_{k-2}\cdots T_0P_0$, we see that
\begin{equation}\label{es-lm}
 \lambda_k\le \prod_{i=0}^k(1+\theta \mu_i)\le e^{\theta\sum_{i=0}^k\mu_i}.
\end{equation}
Observe   that
 $$B_{C_1h^{1/2+\epsilon}}\subset D_h\subset B_{C_2h^{1/2-\epsilon}},$$
where the constant $C_1$, $C_2$ depend on the $C^{1,1-\epsilon}$ norm of $u$ and the convexity of $u$ and the domains. Since $2^kQ_kD_{4^{-k}}[u](0)\sim B_1$,
we can bound the norm of $Q_k$ and $Q_k^{-1}$ by $(4^{-k})^{1-\epsilon}.$ Then by direct computation,
$$\begin{aligned}
\epsilon_k&=[D\rho_k]_{C^{0,1}(P_{x'}(2^kQ_kD_{4^{-k}}))}\\
&\le 2^{-k}||Q_k||\cdot ||Q_k^{-1}||^2\cdot [D\rho_0]_{C^{0,1}(P_{x'}D_ 1)}\\
&\le C(4^{-k})^{1/2-3\epsilon}\epsilon_0.
\end{aligned}
$$
Similarly we obtain the estimates for $\epsilon_k^*$ and $\gamma_k$:
$$\left\{\begin{aligned}
\epsilon_k &\le C(4^{-k})^{\frac{1}{2}-3\epsilon}\epsilon_0,\\
\epsilon_k^* &\le C(4^{-k})^{\frac{1}{2}-3\epsilon}\epsilon_0^*,\\
\gamma_k &\le C(4^{-k})^{-\frac{5}{2}\epsilon+\epsilon^2}\gamma_0.
\end{aligned}\right.$$
This gives that
\begin{equation}\label{es-mu}
 \mu_j\le\delta_j+C_*(4^{-j})^{1/2-10\epsilon},
\end{equation}
where $C_*$ is universal constant which is big enough such that $C(\epsilon_0^*\gamma_0^2+\epsilon_0^{1-\epsilon}\gamma_0)\le C_*$. putting (3.12) in~\eqref{es-lm}, we obtain the estimate for $\lambda_k$.

Now we estimate $I_3=|Dw_k(z)-Du(z)|$. For $z\in D_{4^{-k-3}}[u](0)$, there is a universal constant $l_0$ and $C>0$ such that $$D_{4^{-k-l_0}}[u](z)\subset D_{4^{-k}}[u](0)\subset CD_{4^{-k-l_0}}[u](z).$$
We have
\begin{equation}\label{ztri}
 |Dw_k(z)-Du(z)|\le |Dw_k(z)-Dw_{k+l_0,z}(z)|+|Dw_{k+l_0,z}(z)-Du(z)|,
\end{equation}
where $w_{j,z}$ is defined by
$$\left\{\begin{aligned}
&\det D^2 w_{j,z}=f(z)\quad \text{ in }D_{4^{-j}}[u](z),\\
&w_{j,z}=4^{-j}+u(z)+<Du(z), -z>\quad \text{ on }\partial D_{4^{-j}}[u](z).
\end{aligned}\right.$$
Let
$$\tilde{u}(\tilde{x})=u(A^{-1}\tilde{x}+z)-u(z)-<Du(z),A^{-1}\tilde{x}>,$$
and
$$\tilde{w}_{j,z}(\tilde{x})=w_{j,z}(A^{-1}\tilde{x}+z)-u(z)-<Du(z),A^{-1}\tilde{x}>,$$
where $A$ is constructed as in Lemma~\ref{lem1}. Then $Dw_{j,z}(z)=A^{t}D\tilde{w}_{j,z}(0)+Du(z)$. Now $\tilde{w}_{j,z}$ satisfy
$$\left\{\begin{aligned}
&\det D^2 \tilde{w}_{j,z}=f(z)\quad \text{ in }D_{4^{-j}}[\tilde{u}](0),\\
&\tilde{w}_{j,z}=4^{-j}\quad \text{ on }\partial D_{4^{-j}}[\tilde{u}](0),
\end{aligned}\right.$$
and after this transformation, the equation and domains still satisfy the conditions of Lemma~\ref{lem-c}. Thus we have as before
$$|D\tilde{w}_{j,z}(0)-D\tilde{u}(0)|\le C\sum_{s=j}^{+\infty}2^{-s}\tilde{\lambda}_s\tilde{\mu}_s.$$
Here $\tilde{\lambda}_s$ and $\tilde{\mu}_s$ are the corresponding constants at point $z$ as in inequality~\eqref{cse}. Notice that the norms of $A$ and $A^{-1}$ have global lower and upper bounds. This means that we have
\begin{equation}\label{es-z1}
 |Dw_{j,z}(z)-Du(z)|\le C\sum_{s=j}^{+\infty}2^{-s}\tilde{\lambda}_s\tilde{\mu}_s.
\end{equation}
From~\eqref{es-lm} and~\eqref{es-mu}, we have
\begin{equation}\label{es-z2}
 \begin{split}
   \tilde{\lambda}_s&\le e^{\theta\sum_{j=0}^s\tilde{\mu}_j},\\
   \tilde{\mu}_j&\le \delta_j+C_*(4^{-j})^{1/2-10\epsilon}.
 \end{split}
\end{equation}
Hence the term $|Dw_{k+l_0,z}(z)-Du(z)|$ in~\eqref{ztri} obeys the same estimation in $I_2$.

To estimate $I_3$, it suffices to establish a bound for $|Dw_k(z)-Dw_{k+l_0,z}(z)|$. Assume that $P_{k+l_0,z}$ normalizes $D_{4^{-k-l_0}}[u](z)$.
Let $Q=2^{k+l_0}P_{k+l_0,z}$. Then $QD_{4^{-k-l_0}}[u](z)\sim B_1(z)$, and also from~\eqref{es-xz}, $QD_{4^{-k}}[u](0)\sim B_1(z)$.

Let
$$\hat{w}_k=(\det Q)^{2/n}w_k\circ Q^{-1},\,\hat{w}_{k+l_0,z}=(\det Q)^{2/n}w_{k+l_0,z}\circ Q^{-1},\,\hat{u}=(\det Q)^{2/n}u\circ Q^{-1}.$$
Note that $\det Q=2^{n(k+l_0)}$. Then we have by Lemma~\ref{lem-c},
$$\begin{gathered}
 |\hat{u}-\hat{w}_k|\le C\hat{\mu}_k,\\
 |\hat{u}-\hat{w}_{k+l_0,z}|\le C\hat{\mu}_{k+l_0,z},
\end{gathered}$$
and thus
$$|\hat{w}_k-\hat{w}_{k+l_0,z}|\le C(\hat{\mu}_k+\hat{\mu}_{k+l_0,z})\le C(\delta_k+C_*(4^{-k})^{1/2-10\epsilon}).$$

Thus by Lemma~\ref{lem-2}
$$|D\hat{w}_k(z)-D\hat{w}_{k+l_0,z}(z)|\le C(\delta_k+C_*(4^{-k})^{1/2-10\epsilon}),$$
or equivalently
\begin{equation}\label{es-z3}
 |Dw_k(z)-Dw_{k+l_0,z}(z)|\le C(\delta_k+C_*(4^{-k})^{1/2-10\epsilon})\lambda_{k+l_0,z}2^{-k-l_0},
\end{equation}
where $\lambda_{k+l_0,z}$ is the maximum eigenvalue of $P_{k+l_0,z}$, which is also bounded as~\eqref{es-z2}.

From~\eqref{ztri},~\eqref{es-z1},~\eqref{es-z2} and~\eqref{es-z3}, we see that $I_3$ is bounded by $C(I_1+I_2)$. Then
$$|Du(z)-Du(0)|\le C\sum_{i=k}^{+\infty}2^{-i}\lambda_i\mu_i+C|z|(1+\sum_{i=0}^k\lambda_i^2\mu_i),$$
which, together with (3.11), yields
$$|Du(z)-Du(0)|\le C\sum_{i=k}^{+\infty}2^{-i}\mu_ie^{\theta\sum_{j=0}^i\mu_j}+C|z|(1+\sum_{i=0}^k\mu_ie^{2\theta\sum_{j=0}^i\mu_j}).
$$
Substituting (3.12) into the above inequality, we obtain
$$\begin{aligned}
|Du(z)-Du(0)|\le & e^{\theta C_*}\sum_{i=k}^{+\infty}2^{-i}\delta_ie^{\theta\sum_{j=0}^i\delta_j}+C_*e^{\theta C_*}\sum_{i=k}^{+\infty}2^{-2i}4^{10\epsilon i}e^{\theta\sum_{j=0}^i\delta_j}\\
& +C|z|(1+e^{2\theta C_*}\sum_{i=0}^{k}\delta_ie^{2\theta\sum_{j=0}^i\delta_j}+C_*e^{2\theta C_*}\sum_{i=0}^{k}2^{-i}4^{10\epsilon i}e^{2\theta\sum_{j=0}^i\delta_j}).
\end{aligned}
$$
Let $$\nu(t)=\sup_{x\in\partial\Omega}\sup_{D_{t^2}[u](x)\cap\Omega}|f-f(x)|.$$
Suppose that $(1+C_*)e^{\theta C_*}\le C_0$. By the definition of $\delta_j$ we obtain that
\begin{equation}\label{3.17}
\arraycolsep=0pt
\begin{array}{rcccccc}
|Du(z)&-&Du(0)|\le C_0\int_0^{2^{-k}}\nu(t)&e&^{\theta\int_t^1\frac{\nu(s)}{s}ds}dt+C_0\int^{2^{-k}}_0&t&^{1-20\epsilon}e^{\theta\int_t^1\frac{\nu(s)}{s}ds}dt\\
& & &T_1& &T_2& \\
&+&C|z|(1+C_0\int_{2^{-k}}^1\frac{\nu(t)}{t}&e&^{2\theta\int_t^1\frac{\nu(s)}{s}ds}dt+C_0\int_{2^{-k}}^1&t&^{-20\epsilon}e^{2\theta\int_t^1\frac{\nu(s)}{s}ds}dt).\\
& & & T_3& &T_4&
\end{array}
\end{equation}
 Let $\varphi(t)=-\int_{t}^1\frac{\nu(s)}{s}ds$, then $\varphi'(t)=\frac{\nu(t)}{t}$. We claim that $\varphi(t)=o(|\ln t|)$, $t\to 0$.

In fact, $a(t):=\int_t^1\frac{\nu(s)}{s}ds$ is a monotonic function. We only need to prove the claim for $t=l^m$ for all sufficiently big positive integers $m$, where $l$ is a given constant in $(0,1)$.
 Since $a(\tau t)-a(t)\le \nu(t)|\ln \tau |$ for any $\tau\in (0, 1)$, we see that
$a(l^m)-a(l^{m-1})\le \nu(l^{m-1})|\ln l|$, and add them up to obtain
$$a(l^m)\le (\nu(l^{m-1})+\cdots+\nu(l^0))|\ln l|.$$
The right-hand side of the inequality is less than
$(\frac{k}{m}\nu(1)+\frac{m-k}{m}\nu(l^{k-1}))|\ln l^m|$ for any $m\geq k$.
Take $k=[\sqrt m]$, then the multiplier before $|\ln l^m|$ is less than $\frac{1}{\sqrt m-1}\nu(1)+\nu(l^{\sqrt m-2})$, which goes to 0 when $m\to+\infty$. Hence proves the claim.

To estimate $T_1$ in (3.17), take $r=2^{-k}$ and note that $te^{-\theta\varphi(t)}=e^{\ln t-\theta \varphi(t)}\le e^{\ln t+1/2|\ln t|}=\sqrt{t}\to 0$ when $t\to 0$ by our claim. Using integrate by parts we get
$$T_1=\int_0^r t\varphi'(t)e^{-\theta\varphi(t)}dt=\frac{r}{-\theta}e^{-\theta\varphi(r)}+\frac{1}{\theta}\int_0^re^{-\theta\varphi(t)}dt.
$$
For the second term $T_2$ in (3.17), we have
$$T_2=\frac{1}{2-20\epsilon}r^{2-20\epsilon}e^{-\theta\varphi(r)}+\int_0^r\frac{\theta}{2-20\epsilon}t^{2-20\epsilon}\varphi'(t)e^{-\theta\varphi(t)}dt.
$$
Thus $T_2\le CT_1$. Similarly we have
$$T_3=\int_r^1\varphi'(t)e^{-2\theta\varphi(t)}dt=\frac{1}{2\theta}(e^{-2\theta\varphi(r)}-1)\le Ce^{-2\theta\varphi(r)}.
$$
and
$$T_4=\frac{1}{1-20\epsilon}(1-e^{-2\theta\varphi(r)}r^{1-20\epsilon})+\frac{2\theta}{1-20\epsilon}\int_r^1t^{1-20\epsilon}\varphi'(t)e^{-2\theta\varphi(t)}dt,
$$
after which we have $T_4\le C+CT_3$.

Gathering the estimates for $T_1, \cdots,T_4$ together and recalling $Du(0)=0$, we obtain that
$$|Du(z)|\le C_0(\frac{r}{-\theta}e^{-\theta\varphi(r)}+\frac{1}{\theta}\int_0^re^{-\theta\varphi(t)}dt)+C_0|z|(1+e^{-2\theta\varphi(r)}).
$$
Since $\int_0^r e^{-\theta\varphi(t)}dt=O(re^{-\theta\varphi(r)})$, $r=2^{-k}$, we have
$$|Du(z)|\le C_02^{-k}(1+e^{-\theta\varphi(2^{-k})})+C_0|z|(1+e^{-2\theta\varphi(2^{-k})}).
$$

Repeating the arguments from (3.10) to (3.12) in~\cite{jw}, we have
$$|Du(z)|\le C_0|z|(1+e^{-2\theta\varphi(2^{-k})}).
$$
Note that $B_{C_1t^{1/2+\epsilon}}\subset D_{t^2}[u](z)\subset B_{C_2t^{1/2-\epsilon}}$,
then
$$\begin{aligned}
|\varphi(2^{-k})|&=\int_{2^{-k}}^1\frac{\nu(s)}{s}ds\\
&\le \int_{(\frac{d}{C_0})^{\frac{1}{1-\epsilon}}}^1\frac{w_f(C_0t^{1-\epsilon})}{C_0t^{1-\epsilon}}d\frac{C_0t^{1-\epsilon}}{1-\epsilon}\\
&=\frac{1}{1-\epsilon}\int_d^1\frac{w_f(t)}{t}dt.\\
\end{aligned}
$$
In this way, we have proved the formula~\eqref{fm0}. Since we omitted the $\tilde{ }$ over $u$, Theorem~\ref{thm3} is proved for the case $x,y\in \partial\Omega$.

\subsubsection{Step 2}
Now we turn to the proof of inequality~\eqref{es-1} for $x\in\partial\Omega$ and $y\in \Omega$.

W.l.o.g. $0\in\partial \Omega$, $u(0)=0$, $Du(0)=0$. And the inner unit normal vectors of $\partial\Omega$ and $\partial\Omega^*$ at $0$ are both $e_n$.

First assume that $z\in \Omega\cap(D_{4^{-k-3}}[u](0)\backslash D_{4^{-k-4}}[u](0))$, and $dist(z,\partial\Omega)=|z-0|$. Then by the convexity of $\partial\Omega$, $z$ must locate on the $x_n$ axis. Or otherwise suppose that the line that is parallel to $e_n$ and passes through $z$ intersects $\partial\Omega$ by $z'\neq 0$, then $dist(z,\partial\Omega)\le|z-z'|<|z|$, which results in a contradiction.

Similar to Step 1, we have
$$|Du(z)-Du(0)|\le |Dw_k(z)-Dw_k(0)|+|Dw_k(0)-Du(0)|+|Dw_k(z)-Du(z)|.$$
The first two terms on the right-hand side submit to the same bounds of $I_1$ and $I_2$ as before.

To estimate $I_3$ in this case, first observe that there exists a positive integer $l_0$ such that
\begin{equation}\label{cc}
 S_{4^{-k-l_0}}[u](z)\subset D_{4^{-k}}[u](0)\subset CS_{4^{-k-l_0}}[u](z),
\end{equation}
in which case $S_{4^{-k-l_0}}[u](z)=\{x\in \mathbb{R}^n:u(x)\le u(z)+<Du(z),x-z>+4^{-k-l_0}\}$ is contained in $\Omega$.
To see this, note that by Lemma 2.2 in~\cite{ca3}, there are some constants $C$ and $\delta>0$ that for any $h$ sufficiently small,
$$S^c_{h\delta}[u](z)\subset S^c_{h}[u](0)\subset CS^c_{h\delta}[u](z)$$
for $z\in S^c_{1/4h}[u](0)$. Combining this formula with (2) in Proposition~\ref{prop} and Lemma 2.2 in~\cite{clw}, we obtain~\eqref{cc}.
To show that $S_{4^{-k-l_0}}[u](z)\subset \Omega$, take big enough $l_0$ and small $h_0$ in the beginning of this section, then $S_{4^{-k-l_0}}[u](z)\subset B_{C4^{(-k-l_0)(1/2-\epsilon)}}(z)\subset B_{r_0}(z')$, where $z'$ is on $x_n$ axis and $B_{r_0}(z')$ is contained in $\Omega$ for such $h_0$.

Define by $w_{j,z}$ the solution of the equation
$$\left\{\begin{aligned}
 &\det D^2 w_{j,z}=f(z)\quad \text{ in }S_{4^{-j}}[u](z),\\
 &w_{j,z}=4^{-j}+u(z)+<Du(z),\cdot-z>\quad \text{ on }\partial S_{4^{-j}}[u](z).
\end{aligned}\right.$$
for sufficiently big $j$. Then by (3.4) in~\cite{jw},
\begin{equation}\label{es--z1}
 |Dw_{j,z}(z)-Du(z)|\le C\sum_{s=j}^{+\infty}2^{-s}\Lambda_s\nu_s
\end{equation}
where $\Lambda_s$ is the maximum eigenvalue of $T_s$ which normalizes $S_{4^{-s}}(z)$, and $\nu_s=\sup_{S_{4^{-s}(z)}}|f-f(z)|$. By~\eqref{es-z3},
\begin{equation}\label{es--z2}
 |Dw_{k+l_0,z}(z)-Dw_k(z)|\le C\Lambda_{k+l_0}2^{-k-l_0}(\nu_k+\delta_k+C_*(4^{-k})^{1/2-10\epsilon})
\end{equation}
By (3.8)-(3.9) in~\cite{jw} we have
$$\Lambda_i\le e^{\theta\sum_{j=0}^i}\nu_j$$
Recalling the definition of $\delta_k$ and $\nu_k$, if we let
$$\nu(t)=\sup_{x\in \overline{\Omega}}\sup_{S^c_{t^2}[u](x)}|f-f(x)|$$
then $\delta_k+\nu_k\le \nu(C2^{-k})$. This fact together with~\eqref{es--z1} and~\eqref{es--z2} shows that $|Dw_k(z)-Du(z)|$ for $z\in\Omega$, as well as $|Du(z)-Du(0)|$, admits the same estimate as before.

In all, we have proved that
\begin{equation}\label{esd}
 |Du(z)-Du(0)|\le Ch(|z|)
\end{equation}
for $h(d)=d(1+e^{-2\theta\psi(d)})$ when $dist(z,\partial\Omega)=|z|$.

Claim that this actually yields~\eqref{es-1} for $x,y\in\overline{\Omega}$. First note that $dist(z,\partial\Omega)\le |z|$ for any $z\in\Omega\cap B_1(0)$. For $h$ defined as before, $h$ is monotonically increasing, $h(Cd)\le Ch(d)$, and $h$ is concave for sufficiently small $d$. To show this, we have
$$h'(d)=1+(1-2\theta w_{\log f}(d))e^{-2\theta\psi(d)},$$
which is positive when $d$ is small. One can also verify its concavity by direct computation. Hence for any $z\in D_{1/4}[u](0)\cap \Omega$, denoting by $\hat{z}$ the point on $\partial \Omega$ that has the least distance with $z$, we have
$$\begin{aligned}
|Du(z)-Du(0)|&\le |Du(z)-Du(\hat{z})|+|Du(\hat{z})-Du(0)|\\
&\le Ch(|z-\hat{z}|)+Ch(|\hat{z}|)\\
&\le Ch(\frac{|z-\hat{z}|+|\hat{z}|}{2})\\
&\le Ch(\frac{3}{2}|z|)\le Ch(|z|)
\end{aligned}$$
by~\eqref{esd},~\eqref{fm0} and properties of $h$. Thus we have prove Theorem~\ref{thm3} for $x\in\partial\Omega$ and $y\in\Omega$.

By the triangle inequality and the results in Steps 1 and 2,~\eqref{es-1} is also true for $x,y\in\Omega$. Therefore the proof for Theorem~\ref{thm3} is now complete.

\end{document}